\documentclass[12pt]{article}
\usepackage{amssymb}
\usepackage{amsfonts}
\usepackage{amsmath}
\usepackage{setspace}

\newtheorem{theorem}{Theorem}

\newtheorem{lemma}[theorem]{Lemma}

\newtheorem{proposition}[theorem]{Proposition}

\newenvironment{proof}[1][Proof]{\noindent\textbf{#1.} }{\ \rule{0.5em}{0.5em}}
\input{tcilatex}

\begin{document}

\title{Minimal Generators of Annihilators of Neat Even Elements in the
Exterior Algebra}
\author{Song\"{u}l Esin \\
19 May\i s mah. Tuccarbasi sok. No:10A/25 Istanbul, Turkey\\
songulesin@gmail.com$\medskip $\\
Dedicated to the memory of Professor Cemal Ko\c{c}}
\maketitle

\begin{abstract}
In this paper we exhibit a minimal set of generators form the annihilator of
even neat elements of the exterior algebra of a vector space, when the base
field is of positive characteristic and thus we prove the conjecture we
established in [3]. In order to do that we heavily use the results obtained
in [1] and [2]. This also allows us to exhibit a vector space basis for the
annihilators under consideration.
\end{abstract}

\footnotetext{%
}

\section{Introduction}

Let $V$ be a finite dimensional vector space over a field $F$, and let $E(V)$
be the exterior algebra on $V$. An element $\xi \in E(V)$ is called a
decomposable $m-$vector if $\xi =x_{1}\wedge x_{2}\wedge \cdots \wedge x_{m}$
for some $x_{1},x_{2},\cdots ,x_{m}\in V.$ A sum $\mu =\xi _{1}+\xi
_{2}+\cdots +\xi _{s}$ of decomposable elements of $E(V)$ is said to be 
\textbf{neat} if $\xi _{1}\wedge \xi _{2}\wedge \cdots \wedge \xi _{s}\neq
0. $ This amounts to say that if we let $M_{k}=\{x_{k1},x_{k2},\cdots
,x_{kn_{k}}\}$ the set of factors of $\xi _{k}=x_{k1}\wedge x_{k2}\wedge
\cdots \wedge x_{kn_{k}}$ for each $k\in \{1,2,\ldots ,s\},$ then $%
M=\tbigcup\limits_{k=1}^{s}M_{k}$ is linearly independent.

In \cite{KE}, we proved that when $F$ is a field of characteristic zero the
annihilator of even neat element $\mu $, i.e, $n_{1},n_{2},\ldots ,n_{s}$
are all even,%
\begin{equation*}
(\xi _{i_{1}}-\xi _{j_{1}})\cdots (\xi _{i_{r}}-\xi _{j_{r}})u_{k_{1}\cdots
}u_{k_{t}}
\end{equation*}%
where\ $u_{k_{l}}\in M_{k_{l}}$ and $\{i_{1},\ldots ,i_{r};j_{1},\ldots
,j_{r};k_{1},\ldots ,k_{t}\}=\{1,2,\ldots ,s\}$ with $2r+t=s$. There we also
conjectured that this result is true for any base field of characteristic $\
p>\frac{s+1}{2}.$

\bigskip

The aim of this paper is three-fold: For an even neat element $\mu $

\begin{description}
\item[ \ (i)] to prove this conjecture,

\item[ (ii)] to determine minimal generators of the annihilator of $\mu $
for all characteristics,

\item[(iii)] to describe the vector space structure of both the principal
ideal $(\mu )$ and its annihilator $Ann(\mu )$ in $E(V)$ by using
stack-sortable polynomials introduced in \cite{KE} and the results of \cite%
{KGE}$.$
\end{description}

\bigskip

To begin with we give the notations that will be used throughout the text.

\bigskip

\textbf{Notations:}$\medskip $

$S=\{1,\ldots ,s\}\medskip $

$\Gamma :$ the ideal of the polynomial ring $F[x_{1},\ldots ,x_{s}]$
generated by $x_{1}^{2},\ldots ,x_{s}^{2}.\medskip $

$A=F[x_{1},\ldots ,x_{s}]/\Gamma =F[\xi _{1},\ldots ,\xi _{s}]$ \ with $\
\xi _{i}=x_{i}+\Gamma .\medskip $

$A_{k}:$~the homogeneous component of $A$ consisting of homogeneous
ele\medskip \linebreak \medskip\ \ \ \qquad\ ments of degree $k.\medskip $

$\mu =\xi _{1}+\cdots +\xi _{s}$ $\medskip $

$\mathcal{M}_{t}:$~the set of $t$-$th$ degree monomials $M_{K}=\xi
_{k_{1}}\cdots \xi _{k_{t}}$ where\medskip \linebreak \medskip $\qquad
\qquad K=\{k_{1},\ldots ,k_{t}\}\subset \{1,\ldots ,s\}.\medskip $

$\mathcal{G}_{k}:$ the set of elements of the form 
\begin{equation*}
\gamma _{I,J}=(\xi _{i_{1}}-\xi _{j_{1}})(\xi _{i_{2}}-\xi _{j_{2}})\cdots
(\xi _{i_{k}}-\xi _{j_{k}})
\end{equation*}

\qquad where $I=\{i_{1},\ldots ,i_{k}\}$and $J=\{j_{1},\ldots ,j_{k}\}$ are
disjoint subsets of $S.\medskip $

$\mathcal{G}_{S}:$ the set of all products of the form $M_{K}\gamma _{I,J}$
with $I\cup J\cup K=S$ and \smallskip \newline

\qquad $|I|+|J|+|K|=|S|.\medskip $

$\mathcal{P}_{S}:$ the set of stack-sortable polynomials%
\begin{equation*}
\text{ \ \ \ \ }(\zeta _{\sigma (1)}-\eta _{1})\cdots (\zeta _{\sigma
(d)}-\eta _{d})\text{ with }\sigma \text{ is a stack-sortable permutation,}
\end{equation*}

\qquad where $\eta _{k}=\xi _{2k},~\zeta _{k}=\xi _{2k-1}$ \ and \ $\eta
_{d}=0$ \ when $s$ is odd, \cite{KGE}.

\section{\textbf{Frobenius Algebra Structure of Multilinear Polynomials}}

Recall that a finite dimensional algebra is called a \textit{Frobenius
algebra} if there is a \textit{nondegenerate bilinear form} $B$ satisfying
the associativity condition\ $B(ab,c)=B(a,bc)$ for all elements of the
algebra. Further if $B$ is also symmetric the algebra is called a \textit{%
symmetric algebra. }The exterior algebra is an important example of
Frobenius algebras \cite{KE}.\textit{\ }First of all we note that $A$ is a
symmetric algebra.

\begin{lemma}
\label{Frobenius}$A=F[\xi _{1},\ldots ,\xi _{s}]$ is a symmetric algebra,
i.e. it has a nondegenerate symmetric bilinear form 
\begin{equation*}
B:A\times A\rightarrow F\text{ \ \ such that }B(ab,c)=B(a,bc)
\end{equation*}
\end{lemma}

\begin{proof}
Let $\phi $ be the linear form on $A$ sending each $a\in A$ to its leading
coefficient i.e. the coefficient of the monomial $\xi _{1}\ldots \xi _{s}$
in the expression for $a.$ Then%
\begin{equation*}
B(a,b)=\phi (ab)
\end{equation*}%
provides a bilinear form whose matrix relative to the standard basis of
monomials is a permutation matrix since%
\begin{equation*}
B(\xi _{i_{1}}\cdots \xi _{i_{k}},\xi _{j_{1}}\cdots \xi _{j_{l}})=\left\{ 
\begin{array}{ccc}
1 & \text{if} & \{i_{1}\cdots i_{k,}j_{1}\cdots j_{l}\}=S \\ 
0 & \text{ } & \text{otherwise}%
\end{array}%
\right. .
\end{equation*}%
The last part follows at once from the associativity in $A.$
\end{proof}

In the subsequent sections we will see that $A$ can be regarded as a
subalgebra of an exterior algebra $E$ which inherits its Frobenius property
and we will determine generators of $Ann_{E}(\mu )$ by using $Ann_{A}(\mu )$.

\section{Annihilators of Principal Ideals of the Exterior Algebra}

In this section we will describe generators of $Ann_{A}(\mu )$ under the
assumption that $Char(F)>\frac{s}{2}.$ Also, we note that for any even
elements $\xi _{i}$\ and $\xi _{j},$\ we have $\xi _{i}\xi _{j}=\xi _{j}\xi
_{i}.$ First we give the following lemma that use in several proofs in this
paper.

\begin{lemma}
\label{Obvious Lemma}Let $\eta \in \mathcal{G}_{S}.$ Then $\mu \eta =0.$
\end{lemma}

\begin{proof}
Since $\eta \in \mathcal{G}_{S},$ we write $\eta =\xi _{k_{1}}\cdots \xi
_{k_{t}}(\xi _{i_{1}}-\xi _{j_{1}})(\xi _{i_{2}}-\xi _{j_{2}})\cdots (\xi
_{i_{k}}-\xi _{j_{k}})$ where $\{i_{1},\ldots ,i_{k};j_{1},\ldots
,j_{k};k_{1},\ldots ,k_{t}\}=\{1,\ldots ,s\}=S.$ 
\begin{eqnarray*}
\mu \eta &=&(\xi _{1}+\cdots +\xi _{s})\xi _{k_{1}}\cdots \xi _{k_{t}}(\xi
_{i_{1}}-\xi _{j_{1}})(\xi _{i_{2}}-\xi _{j_{2}})\cdots (\xi _{i_{k}}-\xi
_{j_{k}}) \\
&=&((\xi _{i_{1}}+\xi _{j_{1}})+\cdots +(\xi _{i_{k}}+\xi _{j_{k}}))(\xi
_{i_{1}}-\xi _{j_{1}})(\xi _{i_{2}}-\xi _{j_{2}})\cdots (\xi _{i_{k}}-\xi
_{j_{k}})
\end{eqnarray*}%
\begin{equation*}
\text{ \ \ \ \ \ \ \ \ \ \ \ \ \ \ \ \ \ \ \ \ \ \ \ \ \ \ \ \ \ \ \ \ \ \ \
\ \ since \ }\xi _{k_{r}}\xi _{k_{r}}=0\text{ \ for all }r=1,2,\ldots t.
\end{equation*}%
Also, for all $r=1,2,\ldots k,$ $(\xi _{i_{r}}+\xi _{j_{r}})(\xi
_{i_{r}}-\xi _{j_{r}})=\xi _{i_{r}}\xi _{j_{r}}-\xi _{j_{r}}\xi _{i_{r}}=0$.
Then we get $\mu \eta =0.$
\end{proof}

A straightforward calculation gives the following:

\begin{lemma}
\label{Lemma1}Let $\omega \in A_{k}$ be such that $\omega =\omega
_{0}+\omega _{1}\xi _{s}$ \ with $\omega _{0},\omega _{1}\in F[\xi
_{1},\ldots ,\xi _{s-1}]$ and%
\begin{equation*}
\omega _{0}=\omega _{0}^{\prime }(\xi _{1}+\cdots +\xi _{s-1})\text{ \ and \ 
}\omega _{1}-\omega _{0}^{\prime }=\omega _{1}^{\prime }(\xi _{1}+\cdots
+\xi _{s-1})
\end{equation*}%
then $\omega =\omega ^{\prime }\mu $ for some $\omega ^{\prime }\in F[\xi
_{1},\ldots ,\xi _{s}].$
\end{lemma}

\begin{lemma}
\label{Lemma2}If $Char(F)>k>\frac{s}{2},$ then $\mathcal{M}_{k}\subset A\mu
. $
\end{lemma}

\begin{proof}
Since $s<2k,$ it follows from $s-k<k$ that 
\begin{equation*}
(\xi _{j_{1}}+\cdots +\xi _{j_{s-k}})^{k}=0
\end{equation*}%
for any $J=\{j_{1},\ldots ,j_{s-k}\}\subset S.$ Therefore for each $%
I=\{i_{1},\ldots ,i_{k}\}$ by considering its complement $J=\{j_{1},\ldots
,j_{s-k}\}$ in $S$, we obtain 
\begin{eqnarray*}
\xi _{i_{1}}\cdots \xi _{i_{k}} &=&\frac{1}{k!}(\xi _{i_{1}}+\cdots +\xi
_{i_{k}})^{k} \\
&=&\frac{1}{k!}(\xi _{i_{1}}+\cdots +\xi _{i_{k}})^{k}-(-1)^{k}\frac{1}{k!}%
(\xi _{j_{1}}+\cdots +\xi _{j_{s-k}})^{k} \\
&=&\frac{1}{k!}(\xi _{i_{1}}+\cdots +\xi _{i_{k}}+\xi _{j_{1}}+\cdots +\xi
_{j_{s-k}})\beta \\
&=&\frac{1}{k!}(\xi _{1}+\cdots +\xi _{s})\beta
\end{eqnarray*}%
where $I\cup J\cup K=S$ \ for some $\beta \in A.$
\end{proof}

\begin{proposition}
\label{Prop1}If $\omega \in A_{k}$ annihilates all elements of the form 
\begin{equation*}
\xi _{k_{1}}\cdots \xi _{k_{t}}(\xi _{i_{1}}-\xi _{j_{1}})(\xi _{i_{2}}-\xi
_{j_{2}})\cdots (\xi _{i_{k}}-\xi _{j_{k}})\text{\ }
\end{equation*}%
then $\omega \in A_{k-1}\mu .$
\end{proposition}

\begin{proof}
We use induction on the pairs $(s,k)$ \ with $2k\leq s$ considered with
lexicographic ordering. The case $(s,1)$ is trivial. Suppose the assertion
is true for all $(s^{\prime },k^{\prime })<(s,k).$ Let%
\begin{equation*}
\omega \xi _{k_{1}}\cdots \xi _{k_{t}}(\xi _{i_{1}}-\xi _{j_{1}})(\xi
_{i_{2}}-\xi _{j_{2}})\cdots (\xi _{i_{k}}-\xi _{j_{k}})=0\text{ \ when }%
I\cup J\cup K=S.
\end{equation*}%
There are two cases to consider.\smallskip

\textbf{Case 1: }$s>2k.$ In this case by writing $\omega =\omega _{0}+\omega
_{1}\xi _{s}$ \ with $\omega _{0},\omega _{1}\in F[\xi _{1},\ldots ,\xi
_{s-1}]$ we obtain 
\begin{eqnarray*}
\omega \xi _{k_{1}}\cdots \xi _{k_{t-1}}\xi _{s}(\xi _{i_{1}}-\xi
_{j_{1}})(\xi _{i_{2}}-\xi _{j_{2}})\cdots (\xi _{i_{k}}-\xi _{j_{k}}) &=&0
\\
(\omega _{0}+\omega _{1}\xi _{s})\xi _{k_{1}}\cdots \xi _{k_{t-1}}\xi
_{s}(\xi _{i_{1}}-\xi _{j_{1}})(\xi _{i_{2}}-\xi _{j_{2}})\cdots (\xi
_{i_{k}}-\xi _{j_{k}}) &=&0 \\
\omega _{0}\xi _{k_{1}}\cdots \xi _{k_{t-1}}\xi _{s}(\xi _{i_{1}}-\xi
_{j_{1}})(\xi _{i_{2}}-\xi _{j_{2}})\cdots (\xi _{i_{k}}-\xi _{j_{k}}) &=&0
\\
\omega _{0}\xi _{k_{1}}\cdots \xi _{k_{t-1}}(\xi _{i_{1}}-\xi _{j_{1}})(\xi
_{i_{2}}-\xi _{j_{2}})\cdots (\xi _{i_{k}}-\xi _{j_{k}}) &=&0
\end{eqnarray*}%
when $I\cup J\cup K=S-\{s\}.$ By the induction hypothesis applied for the
pair $(s-1,k),$ this implies%
\begin{eqnarray*}
\omega _{0} &=&\alpha _{0}(\xi _{1}+\cdots +\xi _{s-1})\text{ \ where }%
\alpha _{0}\in F[\xi _{1},\ldots ,\xi _{s-1}] \\
\omega &=&\alpha _{0}(\xi _{1}+\cdots +\xi _{s-1})+\omega _{1}\xi _{s} \\
&=&\alpha _{0}(\xi _{1}+\cdots +\xi _{s})+(\omega _{1}-\alpha _{0})\xi _{s}
\\
&=&\alpha _{0}\mu +(\omega _{1}-\alpha _{0})\xi _{s}.
\end{eqnarray*}%
Consequently, it follows from 
\begin{eqnarray*}
\omega \xi _{k_{1}}\cdots \xi _{k_{t}}(\xi _{i_{1}}-\xi _{j_{1}})(\xi
_{i_{2}}-\xi _{j_{2}})\cdots (\xi _{i_{k}}-\xi _{j_{k}}) &=&0 \\
(\alpha _{0}\mu +(\omega _{1}-\alpha _{0})\xi _{s})\xi _{k_{1}}\cdots \xi
_{k_{t}}(\xi _{i_{1}}-\xi _{j_{1}})(\xi _{i_{2}}-\xi _{j_{2}})\cdots (\xi
_{i_{k}}-\xi _{j_{k}}) &=&0
\end{eqnarray*}%
when $I\cup J\cup K=S$ that 
\begin{equation*}
(\alpha _{0}\mu +(\omega _{1}-\alpha _{0})\xi _{s})\xi _{k_{1}}\cdots \xi
_{k_{t}}(\xi _{i_{1}}-\xi _{j_{1}})(\xi _{i_{2}}-\xi _{j_{2}})\cdots (\xi
_{i_{k-1}}-\xi _{j_{k-1}})(\xi _{i_{k}}-\xi _{s})=0
\end{equation*}%
when $I\cup J\cup K=S.$ Hence, by Lemma \ref{Obvious Lemma} we get 
\begin{eqnarray*}
(\omega _{1}-\alpha _{0})\xi _{s}\xi _{k_{1}}\cdots \xi _{k_{t}}(\xi
_{i_{1}}-\xi _{j_{1}})(\xi _{i_{2}}-\xi _{j_{2}})\cdots (\xi _{i_{k-1}}-\xi
_{j_{k-1}})\xi _{i_{k}} &=&0 \\
(\omega _{1}-\alpha _{0})\xi _{k_{1}}\cdots \xi _{k_{t}}(\xi _{i_{1}}-\xi
_{j_{1}})(\xi _{i_{2}}-\xi _{j_{2}})\cdots (\xi _{i_{k-1}}-\xi
_{j_{k-1}})\xi _{i_{k}} &=&0
\end{eqnarray*}%
when $I\cup J\cup K=S-\{s\},$ and the induction hypothesis yields%
\begin{equation*}
\omega _{1}-\alpha _{0}=\alpha _{1}(\xi _{1}+\cdots +\xi _{s-1})
\end{equation*}%
and hence $\omega =\omega ^{\prime }\mu $ for some $\omega ^{\prime }\in
F[\xi _{1},\ldots ,\xi _{s}]$ as asserted.\smallskip

\textbf{Case 2: } $s=2k.$ Again we write $\omega =\omega _{0}+\omega _{1}\xi
_{s}$ \ with $\omega _{0}$ and $\omega _{1\text{ }}$in $F[\xi _{1},\ldots
,\xi _{s-1}]$ of degrees $k$ and $k-1$ respectively. Since $s-1<2k,$ by
Lemma \ref{Lemma2}, $\omega _{0}=\alpha _{0}(\xi _{1}+\cdots +\xi _{s-1})$
and therefore%
\begin{eqnarray*}
\omega &=&\omega _{0}+\omega _{1}\xi _{s} \\
&=&\alpha _{0}\mu +(\omega _{1}-\alpha _{0})\xi _{s}.
\end{eqnarray*}%
with $\alpha _{0}$ and $\omega _{1}-\alpha _{0}$ of degree $k-1$ in $F[\xi
_{1},\ldots ,\xi _{s-1}]$ as above. Thus, we have%
\begin{eqnarray*}
(\omega _{1}-\alpha _{0})\xi _{s}(\xi _{i_{1}}-\xi _{j_{1}})(\xi
_{i_{2}}-\xi _{j_{2}})\cdots (\xi _{i_{k-1}}-\xi _{j_{k-1}})\xi _{i_{k}}\xi
_{k_{1}}\cdots \xi _{k_{t}} &=&0 \\
(\omega _{1}-\alpha _{0})(\xi _{i_{1}}-\xi _{j_{1}})(\xi _{i_{2}}-\xi
_{j_{2}})\cdots (\xi _{i_{k-1}}-\xi _{j_{k-1}})\xi _{i_{k}}\xi
_{k_{1}}\cdots \xi _{k_{t}} &=&0
\end{eqnarray*}%
and it yields that%
\begin{equation*}
\omega _{1}-\alpha _{0}=\beta _{0}(\xi _{1}+\cdots +\xi _{s-1})
\end{equation*}%
by induction applied to the pair $(s-1,k-1).$ By Lemma \ref{Lemma1}, $\omega
=\omega ^{\prime }\mu $ and the proof is completed.
\end{proof}

\bigskip Now we can establish the following theorem.

\begin{theorem}
\label{Theo 6}The notation being as above we have $Ann(\mathcal{G}_{S}%
\mathcal{)=}A\mu $ and hence%
\begin{equation*}
A\mathcal{G}_{S}=Ann(A\mu )\text{ \ and }\dim (A\mathcal{G}_{S})+\dim (A\mu
)=\dim (A)=2^{s}
\end{equation*}
\end{theorem}

\begin{proof}
The inclusion $A\mu \subset Ann(\mathcal{G}_{S}\mathcal{)}$ is obvious. It
remains to prove $Ann(\mathcal{G}_{S}\mathcal{)\subset }A\mu .$ To this end
take any $\omega \in Ann(\mathcal{G}_{S}\mathcal{)}$ of degree $k$ and show
that $\omega \in A\mu .$\ If $k>\frac{s}{2}$ the assertion follows from
Lemma \ref{Lemma2}. In the case $k\leq \frac{s}{2}$, in turn, letting $%
t=s-2k $ we can write%
\begin{equation*}
\omega \xi _{k_{1}}\cdots \xi _{k_{t}}(\xi _{i_{1}}-\xi _{j_{1}})(\xi
_{i_{2}}-\xi _{j_{2}})\cdots (\xi _{i_{k}}-\xi _{j_{k}})=0\text{\ when }%
I\cup J\cup K=S
\end{equation*}%
and the result follows from Proposition \ref{Prop1}.

Finally, by \cite{GK} every element in $\mathcal{G}_{S}$ is a linear
combination of stack-sortable polynomials. Therefore%
\begin{equation*}
Ann(A\mu )=A\mathcal{P}_{S}.
\end{equation*}%
In order to exhibit a basis for $\mathcal{I}=Ann(A\mu ),$ we note first of
all that by the theorem we just proved this ideal is a graded ideal, say%
\begin{equation*}
\mathcal{I}=\mathcal{I}_{d}\oplus \cdots \oplus \mathcal{I}_{s}
\end{equation*}%
where $d$ is the integral part of $\dfrac{s+1}{2}.$ \ give explicitly a
basis for $\mathcal{I}_{m}$ and that $\dim (\mathcal{I})=\dbinom{s}{s-d}.$
\end{proof}

\section{An F-Basis for Annihilators of Principal Ideals}

In order to simplify the presentation, in addition to the notations given in
the introduction we also introduce the following notations:

\bigskip

$V=V_{1}\oplus \cdots \oplus V_{s}:$ an $F-$space of dimension $n$ with a
direct sum decomposition$\medskip $

$X_{k}=\{x_{k1},\ldots ,x_{kn_{k}}\}:$ a basis for $V_{k}\medskip $

$X=\bigcup\limits_{k=1}^{s}X_{k},$ a basis for $V\medskip $

$E=E(V):$ the exterior algebra on $V\medskip $

$\xi _{k}=x_{k1}\ldots x_{kn_{k}}\medskip $ \ and $n_{t}~$'s are all even

$\mu =\xi _{1}+\cdots +\xi _{s}\medskip $

$A=F[\xi _{1},\ldots ,\xi _{s}]$ \ as a subalgebra of $E\medskip $

$p_{k}:$ a product of elements of $X_{k}$ different from $1$ and $\xi
_{k}\medskip $

$P_{K}=$ The set of products of the form $p_{k_{1}\cdots }p_{k_{t}}$ \ for $%
K=\{k_{1},\ldots ,k_{t}\}\subset S\medskip $

$\mathcal{G}_{K}^{^{\prime }}=\mathcal{G}_{S-K}P_{K}\medskip $

$\mathcal{A}:=A\mathcal{G}_{K}^{^{\prime }}\medskip $

\bigskip

Now we can state our main result.$\medskip $

\begin{theorem}
The notation being as above, the annihilator of $\mu $ in $E$ is the ideal $%
\mathcal{A}$ generated by elements of the form%
\begin{equation*}
(\xi _{i_{1}}-\xi _{j_{1}})\cdots (\xi _{i_{r}}-\xi _{j_{r}})u_{k_{1}}\cdots
u_{k_{t}}\text{ }
\end{equation*}%
where $u_{k}\in \{x_{k1},\ldots ,x_{kn_{k}}\}$ and $\{i_{1},\ldots
,i_{r};j_{1},\ldots ,j_{r};k_{1},\ldots ,k_{t}\}=\{1,\ldots ,s\}$ when $%
Char(F)>\frac{s}{2}.$

Further the elements $(\xi _{i_{1}}-\xi _{j_{1}})\cdots (\xi _{i_{r}}-\xi
_{j_{r}})u_{k_{1}}\cdots u_{k_{t}}$ for which 
\begin{equation*}
(\xi _{i_{1}}-\xi _{j_{1}})\cdots (\xi _{i_{r}}-\xi _{j_{r}})\xi
_{k_{1}}\ldots \xi _{k_{t}}\in \{\theta (p^{\tau }(\overline{\xi };\overline{%
\eta }))~|~\tau \in St_{m}^{(2m-s)}\}
\end{equation*}%
(see \cite{KGE}, Section 3) form a minimal generating set.
\end{theorem}

\begin{proof}
First we know that $E$ is a Frobenius algebra (see \cite{KE}). By the Lemma %
\ref{Frobenius} the same is true of $A=F[\xi _{1},\ldots ,\xi _{s}].$
Therefore 
\begin{equation*}
\dim (E)=\dim (E\mu )+\dim Ann_{E}(\mu )\text{ and }\dim (A)=\dim (A\mu
)+\dim Ann_{A}(\mu )
\end{equation*}%
Secondly,$\ \mathcal{A}\subset Ann_{E}(\mu )$ is obvious. Hence it is
sufficient to show that 
\begin{equation*}
\dim (\mathcal{A})=\dim (E)-\dim (E\mu ).
\end{equation*}%
To this end we make use of the direct sum decomposition%
\begin{equation*}
E=\bigoplus\limits_{1\leq l_{1}<\cdots <l_{k}\leq n}Ap_{l_{1}}\cdots
p_{l_{k}}
\end{equation*}%
where each $p_{k}$ is a product of the $x_{kj},$ factors of $\xi _{k}.$ Then
on one hand we have%
\begin{eqnarray*}
E\mu &=&\bigoplus\limits_{1\leq l_{1}<\cdots <l_{k}\leq s}A\mu
p_{l_{1}}\cdots p_{l_{k}} \\
&&\text{ \ } \\
&=&\bigoplus\limits_{1\leq l_{1}<\cdots <l_{k}\leq s}A(\mu -\xi _{l_{1}}-\xi
_{l_{2}}-\cdots -\xi _{l_{k}})p_{l_{1}}\cdots p_{l_{k}},
\end{eqnarray*}%
on the other hand $\mathcal{A}$ is spanned by elements of the form 
\begin{equation*}
(\xi _{i_{1}}-\xi _{j_{1}})\cdots (\xi _{i_{r}}-\xi _{j_{r}})\xi
_{k_{1}}\cdots \xi _{k_{t}}p_{l_{1}}\cdots p_{l_{m}}
\end{equation*}%
that is to say we have 
\begin{equation*}
\mathcal{A}=\bigoplus\limits_{1\leq l_{1}<\cdots <l_{k}\leq s}~\left(
\bigoplus\limits_{p_{l_{1}}\cdots p_{l_{k}}\in P_{L}}A\mathcal{G}%
_{S-L}p_{l_{1}}\cdots p_{l_{k}}\right)
\end{equation*}%
where we assume $p_{l}\neq 1$ and $\xi _{l}.$ Thus elements of $\mathcal{A}$
are linear combinations of \ the products 
\begin{equation*}
(\xi _{i_{1}}-\xi _{j_{1}})\cdots (\xi _{i_{m}}-\xi _{j_{m}})\xi
_{k_{1}}\cdots \xi _{k_{t}}p_{l_{1}}\cdots p_{l_{q}}\text{ }
\end{equation*}%
where the set of indices is equal to $S$ and $p_{l}\neq 1$ and $\xi _{l}.$
In turn, elements of $E\mu $ are linear combinations of%
\begin{equation*}
(\xi _{i_{1}}+\cdots \xi _{i_{m}})\xi _{k_{1}}\cdots \xi
_{k_{t}}p_{l_{1}}\cdots p_{l_{q}}
\end{equation*}%
with the set of indices equal to $S$ again.

Thus, 
\begin{eqnarray*}
\mathcal{A} &=&\bigoplus\limits_{1\leq l_{1}<\cdots <l_{k}\leq s}~\left(
\bigoplus\limits_{p_{l_{1}}\cdots p_{l_{k}}\in P_{L}}(A\mathcal{G}%
_{L^{^{\prime }}})p_{l_{1}}\cdots p_{l_{k}}\right) \\
&&\text{ \ } \\
&=&\bigoplus\limits_{1\leq l_{1}<\cdots <l_{k}\leq s}~\left(
\bigoplus\limits_{p_{l_{1}}\cdots p_{l_{k}}\in P_{L}}(A_{L^{\prime }}%
\mathcal{G}_{L^{^{\prime }}})p_{l_{1}}\cdots p_{l_{k}}\right) \\
&&\text{ \ \ } \\
&=&\bigoplus\limits_{1\leq l_{1}<\cdots <l_{k}\leq s}~\left(
\bigoplus\limits_{p_{l_{1}}\cdots p_{l_{k}}\in P_{L}}Ann_{A_{L^{\prime
}}}(A_{L^{\prime }}\mu _{L^{\prime }})p_{l_{1}}\cdots p_{l_{k}}\right) \text{
\ \ \ \ \ by Theorem \ref{Theo 6}} \\
&&\text{ \ \ \ } \\
&=&\bigoplus\limits_{L\subset S}Ann_{A_{L^{\prime }}}(A_{L^{\prime }}\mu
_{L^{\prime }})P_{L}
\end{eqnarray*}%
and%
\begin{eqnarray*}
E\mu &=&\bigoplus\limits_{1\leq l_{1}<\cdots <l_{k}\leq n}~\left(
\bigoplus\limits_{p_{l_{1}}\cdots p_{l_{k}}\in P_{L}}A\mu p_{l_{1}}\cdots
p_{l_{k}}\right) \\
&&\text{ \ \ } \\
&=&\bigoplus\limits_{1\leq l_{1}<\cdots <l_{k}\leq n}~\left(
\bigoplus\limits_{p_{l_{1}}\cdots p_{l_{k}}\in P_{L}}A_{L^{^{\prime }}}\mu
_{L^{^{\prime }}}p_{l_{1}}\cdots p_{l_{k}}\right) \\
&&\text{ \ \ } \\
&=&\bigoplus\limits_{L\subset S}(A_{L^{^{\prime }}}\mu _{L^{^{\prime
}}})P_{L}.
\end{eqnarray*}%
This yields%
\begin{eqnarray*}
\dim \mathcal{A} &=&\sum\limits_{L\subset S}\dim (Ann_{A_{L^{\prime
}}}(A_{L^{\prime }}\mu _{L^{\prime }}))|P_{L}|\ \ \ \ \ \ \text{and} \\
&&\text{ \ \ } \\
\dim E\mu &=&\sum\limits_{L\subset S}\dim (A_{L^{\prime }}\mu _{L^{\prime
}})|P_{L}|.
\end{eqnarray*}%
Now, on one hand we have by the previous section that 
\begin{equation*}
\dim Ann_{A_{L^{\prime }}}(A_{L^{\prime }}\mu _{L^{\prime }})+\dim
(A_{L^{\prime }}\mu _{L^{\prime }})=\dim A_{L^{\prime }}=2^{|L^{\prime }|}%
\text{.}
\end{equation*}%
On the other hand, we note that the polynomials 
\begin{equation*}
s_{k}=\sum\limits_{\substack{ L\subset S  \\ |L|=k}}z_{l_{1}}\cdots z_{l_{k}}
\end{equation*}%
are elementary symmetric polynomials in $z_{1},\ldots ,z_{s}$ and therefore 
\begin{equation*}
\sum\limits_{k=0}^{s}s_{k}z^{s-k}=(z+z_{1})\cdots (z+z_{s}).
\end{equation*}%
We conclude by letting $\ z_{l}:=2^{n_{l}}-2$ that 
\begin{eqnarray*}
\dim (E\mu )+\dim \mathcal{A} &=&\sum\limits_{L\subset S}2^{|L^{\prime
}|}(2^{n_{l_{1}}}-2)\cdots (2^{n_{l_{_{k}}}}-2) \\
&&\text{ \ } \\
&=&\sum\limits_{k=0}^{s}\sum\limits_{\substack{ L\subset S  \\ |L|=k}}%
2^{|L^{\prime }|}(2^{n_{l_{1}}}-2)\cdots (2^{n_{l_{_{k}}}}-2) \\
&&\text{ \ } \\
&=&\sum\limits_{k=0}^{s}2^{s-k}\sum\limits_{\substack{ L\subset S  \\ |L|=k}}%
z_{l_{1}}\cdots z_{l_{k}} \\
&&\text{ \ } \\
&=&\sum\limits_{k=0}^{s}2^{s-k}s_{k}=(2+z_{1})\cdots (2+z_{s}) \\
&&\text{ \ } \\
&=&2^{n_{1}+\cdots +n_{s}}=2^{n}\medskip
\end{eqnarray*}%
which means that $\mathcal{A}=Ann_{E}(\mu ).$

The last part of the theorem can be obtained by using techniques of \cite%
{KGE} and linearly independence of $\{\theta (p^{\tau }(\overline{\xi };%
\overline{\eta }))~|~\tau \in St_{m}^{(2m-s)}\}$ proved in \cite[Theorem 3]%
{KGE}$.$
\end{proof}

\end{document}